\documentclass[12pt]{article}

\usepackage{amsmath,amsfonts,amssymb,amsthm,epsfig,epstopdf,titling,url,array,mathtools}
\usepackage[a4paper, left=1in, right=1in, top=1in, bottom=1in]{geometry}
\usepackage{hyperref}
\usepackage{setspace}
\usepackage{graphicx}
\usepackage{float}
\usepackage{tikz-cd}
\usepackage{color}
\usepackage{comment}

\theoremstyle{plain}
\newtheorem{thm}{Theorem}[section]
\newtheorem{lem}[thm]{Lemma}
\newtheorem{prop}[thm]{Proposition}
\newtheorem{cor}[thm]{Corollary}

\theoremstyle{definition}
\newtheorem{definition}{Definition}
\newtheorem{con}{Conjecture}

\newtheorem{qstn}{Question}

\theoremstyle{remark}
\newtheorem*{rem}{Remark}
\newtheorem*{rems}{Remarks}

\newtheorem*{notation}{Notation}

\newcommand{\ind}{\mathbf 1}

\newcommand{\R}{\mathbb{R}}
\renewcommand{\d}{\textnormal{ d}}

\newcommand{\lip}{\textnormal{Lip}}

\newcommand{\grad}{\nabla}

\newcommand{\ball}{\mathcal{B}}

\newcommand{\sgn}{\textnormal{sgn}}

\newcommand{\hidethis}[1]{}

\newcommand{\E}{\mathbb{E}}

\newcommand{\prob}{\mathbb{P}}
\newcommand{\essup}{\textnormal{ess sup}}

\newcommand{\supp}{\textnormal{support}}

\newcommand{\vol}{\textnormal{Vol}}

\newcommand{\unif}{\textnormal{Uniform}}
\newcommand{\normal}{\mathcal{N}}

\usepackage{environ}

\makeatletter
\long\def\@secondofthree#1#2#3{#2}
\long\def\@thirdoffour#1#2#3#4{#3}

\NewEnviron{restatethis}[2]{%
  \begin{#1}%
      \BODY
      \protected@write\@auxout{}{%
        \string\@restatetheorem{#1}{#2}{\csname the#1\endcsname}{\detokenize\expandafter{\BODY}}%
      }%
  \end{#1}%
}

\@ifpackageloaded{thmtools}{%
    \def\restatethm@getthmcountercsname#1{\def\thethmcsname{#1}}%
}{%
    \@ifpackageloaded{ntheorem}{%
        \def\restatethm@getthmcountercsname#1{%
            \def\thethmcsname{\expandafter\expandafter\expandafter\restatethm@ntheorem@getthmcountercsname@helper\csname mkheader@#1\endcsname}}%
        \def\restatethm@ntheorem@getthmcountercsname@helper#1\@thm#2#3#4{#3}
    }{%
        \@ifpackageloaded{amsthm}{%
            \def\restatethm@getthmcountercsname#1{\edef\thethmcsname{\expandafter\expandafter\expandafter\@thirdoffour\csname#1\endcsname}}%
        }{%
            \def\restatethm@getthmcountercsname#1{\edef\thethmcsname{\expandafter\expandafter\expandafter\@secondofthree\csname#1\endcsname}}%
        }%
    }%
}

\newcommand{\@restatetheorem}[4]{%
  \expandafter\gdef\csname restatethis@#2\endcsname{%
    \begingroup
    \restatethm@getthmcountercsname{#1}
    \expandafter\def\csname the\thethmcsname\endcsname{#3}%
    \begin{#1}#4\end{#1}%
    \endgroup
  }%
}
\newcommand{\restate}[1]{\csname restatethis@#1\endcsname} 
\makeatother
\begin{document}

\title{\huge Entropic exercises around the Kneser--Poulsen conjecture}
\author{ 
  Gautam Aishwarya\\
  \texttt{Tel Aviv University}\\
  \texttt{gautamaish@gmail.com}
  \and
 Irfan Alam\\
 \texttt{University of Pennsylvania}\\
  \texttt{irfanalamisi@gmail.com}
  \and
 Dongbin Li\\
 \texttt{University of Delaware}\\
 \texttt{leeudel@udel.edu}
  \and
 Sergii Myroshnychenko\\
 \texttt{Lakehead University}\\
  \texttt{smyroshn@lakeheadu.ca}
   \and
 Oscar Zatarain-Vera\\
 \texttt{Centre College}\\
  \texttt{oscar.zatarain-vera@centre.edu}
}
\date{}

\maketitle
\abstract{We develop an information-theoretic approach to study the Kneser--Poulsen conjecture in discrete geometry. This leads us to a broad question regarding whether R\'enyi entropies of independent sums decrease when one of the summands is contracted by a $1$-Lipschitz map. We answer this question affirmatively in various cases. }

\section{Introduction}
If one starts with a finite number of open balls in a Euclidean space, then it appears plausible that the volume of their union should decrease if the centers are rearranged to be pairwise closer. However, this intuition has been incredibly difficult to formalize. Indeed, a proof of the above assertion still eludes us more than six decades since it was first formulated by Poulsen \cite{Poulson54} and Kneser \cite{Kneser55} independently. Before stating the Kneser--Poulsen conjecture formally, let us fix some notation. 

\begin{notation}
Throughout the paper, we will work in the Euclidean space $\mathbb{R}^d$, where $d \in \mathbb{Z}_{>0}$. The $L^2$-norm on $\mathbb{R}^d$ is denoted by $\|\cdot\|_2$, where we will often suppress the subscript and write $\|\cdot\|$ when there is no scope for confusion. Throughout the paper, the metric on $\mathbb{R}^d$ that is used (for instance, in describing Lipschitz functions) is the one induced by $\|\cdot\|$. We use $\vol_d$ (or just $\vol$ when there is no scope of confusion) to denote the Lebesgue measure on $\mathbb{R}^d$. The open ball of radius $r$ centered at the point $x \in \R^d$ will be denoted by $\ball(x,r)$, while the ball $\ball(0,1)$ will be denoted simply by $\ball$. 
\end{notation}

\begin{con}[Kneser--Poulsen]\label{KN_Conj}
Let $\{x_{1} , \ldots , x_{k}\}$ and $\{y_{1}, \cdots , y_{k}\}$ be two sets of points in $\mathbb{R}^d$ such that $\Vert y_{i} - y_{j} \Vert_{2} \leq \Vert x_{i} - x_{j} \Vert_{2}$ for all $i, j \in \{1, \ldots, k\}$. If $r> 0$, then we have:
\begin{equation} \label{KPunion}
\vol_{d} \left( \bigcup_{i=1}^{k} \ball (y_{i}, r) \right) \leq \vol_{d} \left( \bigcup_{i=1}^{k} \ball (x_{i}, r) \right).
\end{equation}
\end{con}

Bezdek and Connelly \cite{BezdekConnelly02} proved this conjecture in the plane. So far $d=2$ remains the only dimension in which the conjecture has been proven completely. For an arbitrary Euclidean space $\R^d$, the conjecture has been proven under additional assumptions on the number of points $k$ and the map $x_{i} \mapsto y_{i}$. For example, Csik\'os \cite{Csiskos98} proved the Kneser--Poulsen conjecture under the assumption that each initial point $x_{i}$ can be joined with the corresponding final point $y_{i}$ by a path such that all pairwise distances decrease along the path. Such a requirement is not always satisfied if the number of points involved exceeds the dimension. More recently, Bezdek and Nasz\'odi \cite{BezdekNaszodi} demonstrated the conjecture for uniform contractions, that is, when there exists $\lambda > 0$ such that $\Vert y_{i} - y_{j} \Vert_{2} < \lambda < \Vert x_{i} - x_{j} \Vert_{2}$ for all $i \neq j$. In the same work, the authors also settle the case when the pairwise distances are reduced in every coordinate. In this quick literature review, we have skipped many rich developments around the conjecture including those pertaining to its formulation in other spaces. The formulation with different radii, as considered by some authors (for example, \cite{BezdekConnelly02, Csiskos98}), is not dealt with in our work. We refer the reader to the book \cite{Bezdek13}, or recent surveys \cite{Csikos18, KuperbergToth22} for a more detailed account. 

The results in our paper are motivated by our attempt to understand the Kneser--Poulsen conjecture from an information-theoretic viewpoint. We begin by gathering the necessary background to formulate our results.

\subsection{Background}
Let $S$ denote the set $\{ x_{1} , \cdots , x_{k} \}$ and let $T \colon \mathbb{R}^d \to \mathbb{R}^d$ be a $1$-Lipschitz map (which we will call a \textit{contraction}) such that 
\begin{align}\label{Kirszbraun}
    y_{i} = T(x_{i}) \text{ for all } i \in \{1, \ldots, k\}.
\end{align}

Given the points $\{y_1, \ldots, y_k\}$ as in the statement of Conjecture \ref{KN_Conj}, a contraction $T$ satisfying \eqref{Kirszbraun} indeed exists by Kirszbraun's theorem. For a fixed $r >0$, observe that the set $\bigcup_{i=1}^{k} \ball (x_{i}, r)$ can be rewritten as $S + r\mathcal{B}$. Using an elementary approximation-from-within argument to go from finite sets to arbitrary compact sets, we can thus rephrase the Kneser-Poulsen conjecture in the following form. 

\begin{con} \label{con: kp}
For every contraction $T$ of $\R^d$ and every compact set $K 
\subseteq \R^d$, $r>0$, we have
\[
\vol(T[K] + r \ball) \leq \vol (K + r \ball).
\]
\end{con}

This reformulation of the Kneser--Poulsen conjecture can be further reinterpreted as a particular case of a broad information-theoretic question. Our paper studies that information-theoretic question, for which we manage to give positive answers to several cases. 

To the best of our knowledge, this method is new in the literature of the Knesen--Poulsen conjecture. We consider this novel approach to be the first steps towards an information-theoretic interpretation of Kneser--Poulsen-type questions in metric geometry. We hope that this connection leads to developments that enrich both fields. For example, our results have straightforward implications for channel capacities of certain additive noise channels. These implications and a channel capacity based approach to Conjecture \ref{con: kp} will be presented in a follow up note. 

Let us begin by establishing some notation needed to express the questions that we aim to study. 

\subsubsection{Some preliminary notation and definitions}
Throughout, unless stated otherwise, all sums $X + Y$ that we study will be sums of independent random vectors. The distribution of a random vector $X$ will sometimes be denoted by $\prob_{X}$, and if it happens to have a density with respect to the Lebesgue measure on the ambient space, its density will sometimes be denoted by $f_{X}$. 

 \begin{definition}
Let $X$ be an $\R^d$-valued random vector with density $f$ with respect to the Lebesgue measure. Then, the R\'enyi entropy of order $\alpha \in (0,1)\cup (1, \infty)$ of $X$ is given by,
\[
h_{\alpha}(X) = \frac{1}{1 - \alpha} \log \int_{\R^d} f^{\alpha}. 
\]
The R\'enyi entropy of orders $0, 1,$ and $\infty$ are obtained via taking respective limits,
\[
\begin{split}
& h_{0}(X) = \log \vol (\supp(f)), \\ 
&h_{1}(X) =  -\int f \log f, \\ 
&h_{\infty}(X) = - \log \Vert f \Vert_{\infty}. \\
\end{split}
\]
The special case $h_{1} (\cdot)$ is called the Shannon-Boltzmann entropy, often denoted simply by $h(\cdot)$. 
\end{definition}

For an $\R^{d}$-valued random vector $X$ with distribution $\mu \in \mathcal{P}(\R^{d})$ having density $f$ with respect to the Lebesgue measure, we will abuse notation and use $h_{\alpha}(X), h_{\alpha} (\mu) ,$ and $h(f)$, interchangeably. 

It has been realized through a continuous stream of works in the past few decades that entropy can be powerfully used as a proxy for volume (and other geometric invariants) to bring concrete geometric problems to the setting of information theory and probability where it becomes susceptible to more analytic tools. This technique has been especially fruitful in the subject of convex geometry (see \cite{MadimanMelbourneXu17}, and the references therein). Information-theoretic analysis is often more effective when the underlying random variables have nice geometric structure. Unsurprisingly, such structures have a convexity flavor. Some, which make an appearance in our results, are defined below.

\begin{definition}
Let $X$ be an $\R^d$-valued random vector with density $f$ with respect to the Lebesgue measure. 
\begin{itemize}
\item $X$ is said to be log-concave if $f = e^{- \phi}$, for some convex function $\phi: \R^d \to (- \infty , \infty ]$. 
\item $X$ is said to be unconditional if its distribution is invariant under reflections about the coordinate axes. 
\item $X$ is unimodal if $\{ f > t \}$ is convex for every $t \in \R$. The terminology ``unimodal'' stems from the fact that the density of a unimodal random vector has a single ``peak''.
\item $X$ is said to be isotropic if its covariance matrix is a scalar multiple of the identity matrix.
\end{itemize}
\end{definition}
We use this opportunity to also define the special classes of contractions that we use in this paper. 
\begin{definition}
Let $T : \R^d \to \R^d$ be a contraction, i.e., a $1$-Lipschitz map between $d$-dimensional standard Euclidean spaces. 
\begin{itemize}
\item $T$ is an affine contraction if $T(x) = A(x) + b$, where $A$ is a linear map and $b \in \R^d$ a fixed vector. Such a $T$ is a linear contraction if $b = 0$. Clearly, the Lipschitz constants of $T$ and $A$ are equal.
\item $T$ is a diagonally linear contraction if $T$ is given by $T(x_{1} , \ldots , x_{d}) = (\lambda_{1} x_{1}, \ldots , \lambda_{d} x_{d})$ for some fixed $\lambda_{1} , \ldots , \lambda_{d} \in \R$. Since $T$ is already assumed to be a contraction, it follows that each $\vert \lambda_{i} \vert \leq 1$.
\item $T$ is said to be a strong contraction if $T = (T_{1} , \ldots , T_{d})$ and $|T_{i}(x) - T_{i} (y) | \leq |x_{i} - y_{i}|$ for all $i = 1 , \ldots , d$ and all $x = (x_{1}, \ldots, x_{d}), y = (y_{1} , \ldots , y_{d}) \in \R^d$. 
\end{itemize}
\end{definition}

\subsubsection{Our information-theoretic question and Kneser--Poulsen-type problems}
Consider the following slightly open-ended question.
\begin{qstn} \label{big question}
Let $X$ and $W$ be $\R^d$-valued random vectors. Further assume that $W$ is log-concave and satisfies a symmetry property such as radial symmetry, or unconditionality, etc. For a contraction $T: \R^d \to \R^d$, and $\alpha \in [0, \infty]$, under what additional assumptions do we have
\[
h_{\alpha}(T(X) + W) \leq h_{\alpha} (X + W ) \, ?
\]
\end{qstn}

Let $K$ be a compact subset of $\mathbb{R}^d$. Let $X$ be a random vector with support $K$ and let $W \sim \unif (\ball)$. Then the support of $X + W$ is $K + \ball$ and the support of $T(X) + W$ is $T[K] + \ball$. If the answer to Question \ref{big question} is true for all $\alpha$ sufficiently close to $0$, then taking limits yields the Conjecture \ref{con: kp} formulation of the Kneser--Poulsen conjecture.  

Let $\alpha = n \geq 2$ be an integer and $W \sim \unif(\ball)$. Suppose $X$ is a discrete random vector taking the values $x_{i}$ with probability $p_{i}$, $i= 1 , \ldots , k$, respectively. Then $X + W$ has density $\frac{1}{\vol(\ball)} \sum p_{i} \ind_{\ball}(x - x_{i})$. Similarly, if $y_{i} = T(x_{i})$ (where $T \colon \mathbb{R}^d \to \mathbb{R}^d$ is a contraction, as before), then the random vector $T(X) + W$ has density $\frac{1}{\vol(\ball)} \sum p_{i} \ind_{\ball}(x - y_{i})$. In this setup, Question \ref{big question} takes the following form  
\[
\int_{\mathbb{R}^d} \left( \frac{1}{\vol(\ball)} \sum_{i=1}^{k} p_{i} \ind_{\ball}(x - y_{i})\right)^n  \d x \geq \int_{\mathbb{R}^d} \left( \frac{1}{\vol(\ball)} \sum_{i=1}^{k} p_{i} \ind_{\ball}(x - x_{i}) \right)^n  \d x \, ?
\]

If one first expands the integrands using the multinomial theorem and then proceeds to a term-by-term comparison, one realizes that an affirmative answer can be obtained if
\[
\vol \left(  \bigcap_{i \in S} \ball(y_{i} , r) \right) \geq \vol \left(  \bigcap_{i \in S}  \ball(x_{i} , r) \right),
\]
for every $S \subseteq \{ 1 , \ldots, k \}$ of cardinality at most $n$. This brings us to a dual formulation of Conjecture \ref{KN_Conj}, first investigated by Gromov \cite{Gromov87} and by Klee and Wagon \cite{KleeWagon91}. 

\begin{con} \label{con: kp int}
Let $\{x_{1} , \ldots , x_{k}\}$ and $\{y_{1}, \cdots , y_{k}\}$ be two sets of points in $\mathbb{R}^d$ such that $\Vert y_{i} - y_{j} \Vert_{2} \leq \Vert x_{i} - x_{j} \Vert_{2}$ for all $i, j \in \{1, \ldots, k\}$. For $r > 0$, we have:
\begin{equation}\label{KPintersect}
\vol \left(  \bigcap_{i=1}^k \ball(y_{i} , r) \right) \geq \vol \left(  \bigcap_{i=1}^k  \ball(x_{i} , r) \right).
\end{equation}
\end{con}

The previous considerations with the multinomial theorem show that the affirmative answer to Conjecture \ref{con: kp int} when the number of points involved is at most $k$,  implies the desired R\'enyi entropic comparisons for integer orders $\alpha = 2 , \cdots , k$.

\begin{prop}
The intersection version of the Kneser-Poulsen conjecture (i.e. Conjecture \ref{con: kp int}) implies, 
\[
h_{k}(T(X) + W) \leq h_{k} (X + W),
\]
for any $\mathbb{R}^d$-valued random vector $X$, where $W \sim \unif (\ball)$ and $k \geq 2$ is an integer.
\end{prop}

Starting from Gromov's work \cite{Gromov87} which established it for at most $d+1$ balls, Conjecture \ref{con: kp int} is now known for at most $d + 3$ balls in $\R^d$ by the work of Bezdek and Connelly \cite[Corollary 4]{BezdekConnelly02}. Thus, we have the following corollary. 
\begin{cor} \label{cor: intenttrue}
Let $X$ be an $\R^{d}$-valued random vector, $T: \R^d \to \R^d$ any contraction and $W \sim \unif (\ball)$. Then for $\alpha = 2, 3 , \ldots , d + 3$, we have
\[
h_{\alpha} (T(X) + W) \leq h_{\alpha} (X + W).
\]
\end{cor}

Returning to the entropic analogues of the (union) Kneser-Poulsen conjecture, we will now sketch the rest of the paper. The main results of this note are compiled below.

\subsection{Description of results}

The main body of the paper begins with several results obtained using the majorization order $\preceq$ on probability densities in $\R^d$ (Definition \ref{def: majorisation}). We prove that the density of $X + W$ is majorized by the density of $T(X) + W$ under several cases. Since R\'enyi entropies are majorization reversing (Lemma \ref{lem: majorisationimpliesconvex}), the desired entropic inequalities are obtained in these cases. Almost all results in this section require log-concavity type assumptions on $X$. The reader will find the reason behind this implicit in the representation \eqref{BurchardRep}, which is heavily used throughout this section. 

The simplest case of a contraction, when $T (x) = \lambda x$, is treated using the Pr\'ekopa-Leindler inequality along with a basic conditioning argument. 

\restate{lambdaX}

Aided by log-concavity and Anderson's theorem, we drop all assumptions on $T$ in exchange for rather strong restrictions on $X$ and $W$. 

\restate{radsymunimodXW}

Next, we let $X$ and $W$ be unconditional log-concave random vectors and consider a diagonally linear contraction $T$. We use a symmetrization technique for convex sets depending on the eigenvalues of $T$ (the keyword here is ``shadow systems'' for the familiar reader), to obtain an affirmative answer to Question \ref{big question} in this regime.  
\restate{lcXunconditionalWdiagT}

By the polar decomposition of an arbitrary linear map $T$ and subsequent diagonalization of its symmetric part, the previous result extends to all linear maps $T$ when the ``noise'' $W$ is unaffected by the orthogonal parts of these decompositions. This is the content of Corollary \ref{cor: lcXradsymWaffineT}, which also implies its geometric counterpart in Corollary \ref{cor: convexKlinearT}. In Corollary \ref{cor: intrinsicvolumeslinearcontractions}, our method allows us to give a geometric proof of the main result of \cite{PaourisPivovarov13}, namely that linear contractions reduce intrinsic volumes of convex sets.
\restate{lcXradsymWaffineT}

Thereafter, a coordinate-wise symmetrization argument implies that we have $f_{X + W} \preceq f_{T(X) + W}$ for any unconditional log-concave random variables $X, W$ when the map $T$ contracts each coordinate. 
\restate{unconditionalXWstrongT}

An immediate geometrical consequence is noted in Corollary \ref{cor: unconditionalconvexKL}. A multidimensional mean-value theorem enables us to derive the conclusion of Theorem \ref{thm: unconditionalXWstrongT} when $W$ is radially-symmetric log-concave for contractions of the form $T = \grad \varphi$, where $\varphi$ is a convex function. This result is presented as the final result of Section \ref{sec: rearrangements}, namely Corollary \ref{cor: unconditionallcXradsymlcWBrennierT}. 

In Section \ref{sec: gaussian noise}, we direct our attention to the special case of Gaussian noise (when $W$ equals a standard Gaussian random vector) and the Shannon-Boltzmann entropy. While this special situation does not have immediate geometric consequences, its importance is paramount in information theory. For us, this setup presents ideal conditions to anticipate results for general $\alpha$ and $W$. Results of this section are stated in the form of Entropy Power $N(\cdot) := e^{\frac{2 h(\cdot)}{d}}$ rather than the Shannon-Boltzmann entropy $h(\cdot)$. 

First, a vector-generalization of Costa's strengthening of the Entropy Power inequality (EPI) is used to obtain a stronger result confirming the desired inequality from Question \ref{big question} in this regime when $T$ is linear. 
\restate{arbitXgaussianZlinearT}

A similar result is obtained when $X = G$ is a Gaussian random vector with independent coordinates and $T$ is a strong contraction. The main idea involved is a (co)variance comparison using the linear algebra of positive semi-definite matrices.
\restate{gaussianGZstrongT}
The observation that the proof of the previous theorem goes through for an arbitrary contraction $T$ if $G$ is isotropic is then generalized to the situation pertaining to an arbitrary isotropic log-concave random vector. 
\restate{isotropiclcXgaussianZ}

Utilizing the relationship between the isotropic constant of $X$ and $\Delta (X)$, an affirmative answer to Question \ref{big question} is obtained as Corollary \ref{cor: isotropiclcXgaussianWsomeT} for log-concave $X$, when $\lip (T)$ is small enough, 
 and when $\alpha=1, W = Z$. Upon going through the proof of Theorem \ref{thm: isotropiclcXgaussianZ}, the reader will notice that the estimates we use are not very tight. We believe that $e^{\Delta (X)}$ can be dropped from the statement of Theorem \ref{thm: isotropiclcXgaussianZ}. Though we are not able to prove it beyond the linear case, we are inclined to believe that the answer to the following strengthening of Question \ref{big question} for $\alpha=1, W=Z$, is affirmative more generally. 
 \begin{qstn}
 Suppose $X$ is a random vector in $\R^d$ and $Z$ a standard Gaussian vector in $\R^d$. Let $T: \R^d \to \R^d$ be a contraction. Do we have
 \[
 N(X + Z) \geq N(T(X) + Z) + (1 - \lip^2 (T)) N(X) ?
 \]
 \end{qstn}

The final section of the paper answers Question \ref{big question} in full generality when $\alpha = 2$ and $W$ is a radially-symmetric log-concave random vector (see Theorem \ref{thm: h2}). We also intend this section to be an advertisement to the growing theory of diversity and maximum diversity in metric spaces (see for example, \cite{LeinsterRoff21, AishwaryaLiMadiman22}). Borrowing language from this theory, we are able to give a rather intuitive proof for Theorem \ref{thm: h2}. The ``diversity of order $2$ at scaling $t$'' (Definition \ref{def: diversity}) is denoted by $D^{t}_{2} (\cdot)$ in the restatement below. 
\restate{h2}
When $X$ is log-concave, a stability result for R\'enyi entropies $h_{\alpha}(X)$ as a function of $\alpha$ is employed to extend Theorem \ref{thm: h2} to all orders. Unfortunately, these estimates in Corollary \ref{cor: lccomparisonfromh2} blow up at $\alpha = 0$, thereby not allowing any direct geometric consequences.

\section{Rearrangement methods} \label{sec: rearrangements}

We will now rearrange the density of random variables $X + W$ and $T(X) + W$ into radially-symmetric unimodal densities while keeping their R\'enyi entropies fixed. As it will be evident, comparing R\'enyi entropies of two radially-symmetric unimodal densities is easier than the general case.

\begin{definition}
For every Borel set $A \subseteq \R^d$ of positive volume, let $A^{\ast}$ denote the centered Euclidean ball in $\R^d$ having the same volume as $A$. Then, for a non-negative measurable $f: \R^d \to [0, \infty)$ which vanishes at infinity, we define its symmetrically-decreasing rearrangement as an almost-everywhere uniquely defined function $f^{\ast}: \R^d \to [0, \infty)$ characterized by the property
\[ 
\{x: f^{\ast}(x)>t \} = \{ x: f(x) >t \}^{\ast},
\]
for every $t>0$.
\end{definition}
\begin{rems}
\begin{enumerate}
\item[]
\item One can describe $f^{\ast}$ explicitly by formula
\begin{equation} \label{f*}
f^{\ast}(x)= \int_{0}^{\infty} \ind_{\{x: f(x) > t \}^\ast}(x) \d t.
\end{equation}
\item As terminology indicates, $f^{\ast}$ is indeed radially-symmetric and decreases radially.
\item The ``layer-cake'' representation for $L^{\alpha}$-norms shows that if $f$ is a probability density, then so is $f^{\ast}$. Moreover, $h_{\alpha}(f) = h_{\alpha}(f^{\ast})$ holds for all $\alpha > 1$. 
\end{enumerate}
\end{rems}
For more information regarding rearrangements, we refer to Lieb and Loss' text \cite[Chapter 3]{LiebLoss01} and Burchard's notes \cite{Burchard09}. 

Recall that we are trying to show $h_{\alpha}(T(X) + W) \leq h_{\alpha}(X + W)$ under various hypotheses. Towards this, we will try to show that $f^{\ast}_{T(X) + W}$ is less spread out than $f^{\ast}_{ X + W}$. We formulate this notion of \textit{spread} using the majorization order. 

\begin{definition} \label{def: majorisation}
For two probability densities $f$ and $g$ on $\R^d$, we say that $f$ is majorized by $g$, written as $f \preceq g$ (or $g \succeq f$), if 
\[ \int_{\ball (0,r)} f^*(x) \d x \leq \int_{ \ball (0,r)} g^*(x) \d x,\]
for all $r>0.$ If $f \preceq g$ and $g \preceq f$, we will write $f \simeq g$.
\end{definition}

Indeed, knowing $f \preceq g$ allows us to conclude that $h_{\alpha}(f) \geq h_{\alpha}(g)$. This can be seen by applying the lemma below to convex functions $\phi (x) = x^{\alpha}$, if $\alpha \geq 1$, and $\phi (x) = - x^{\alpha}$, if $\alpha \leq 1$. 
\begin{lem} \label{lem: majorisationimpliesconvex} \cite[Lemma VII.2.]{MadimanWang14} 
Let $\phi(x)$ be a convex function defined on the non-negative real line such that $\phi(0)=0$ and is continuous at 0. If $f$ and $g$ are probability densities, with $ f \preceq g$, then 
\[ \int_{\R^d} \phi(f(x)) \d x \leq \int_{\R^d} \phi(g(x)) \d x.\]
\end{lem} 

To establish $f \preceq g$ for $f = f_{X + W}$ and $g = f_{T(X) + W}$, it is useful to have a more tractable representation of the integrals involved in Definition \ref{def: majorisation}. The following elementary lemma is suitable for this purpose, which can also be found in \cite{Burchard09}.

\begin{lem} \label{lem: centralintegralrep}
Let $f : \R^d \to [0, \infty)$ be an integrable non-negative function. Then, 
\begin{align}
\int_{\ball (0,r)} f^*(x) \d x = \sup_{\{ C: \vol (C) = \vol (\ball (0,r))\}} \int_C f(x) \d x. \label{BurchardRep}
\end{align}   
Moreover, the supremum is attained by any super-level set $\{ f > t \}$ with the same volume as $\ball (0,r)$. 
\end{lem}
\begin{proof}
By the Hardy-Littlewood inequality (see, for example, \cite[Theorem 3.4]{LiebLoss01}), one has 
\[ \int_C f(x) \d x \leq \int_{C^*} f^*(x) \d x.\]
Whence, we have,
\[ \int_{\ball (0,r)} f^*(x) \d x \geq \sup_{\{ C: \vol (C) = \vol (\ball (0,r))\}} \int_C f(x)\d x.\]
To prove the reversed inequality, for any $r>0$, by the property in Equation \ref{f*}, there exists $t>0$ (depending on $r
$) such that $\ball (0,r)= \{x: f^*(x)>t\}.$ Hence,
\[ \int_{\ball (0,r)} f^*(x) \d x = \int_{\{x: f^*(x)>t\}} f^*(x) \d x = \int_{0}^{\infty} \vol(\{x: f^*(x)>\max\{t,s\}\}) \d s \]
\[= \int_{0}^{\infty} \vol(\{x: f(x)>\max\{t,s\}\}) \d s =  \int_{\{x: f(x)>t \}} f(x) \d x .\]
The reversed inequality now follows.
\end{proof}
Using these pieces, along with an application of the Pr\'ekopa-Leindler inequality, we obtain the first result of this section.

\begin{restatethis}{thm}{lambdaX} 
For any two log-concave random vectors $X$ and $W$, and any $\lambda \in (0,1)$, we have
\[
f_{X + W} \preceq f_{\lambda X + W },
\]
and consequently,
\[
h_{\alpha}(\lambda X + W) \leq h_{\alpha} (X + W) \text{ for all } \alpha \in (0, \infty).
\]
\end{restatethis}
\begin{proof}
Since $X$ and $W$ are log-concave, a direct application of the Pr\'ekopa-Leindler inequality reveals that $X+W$ is also log-concave. Consequently, the super level sets of $X + W$ are bounded and convex. By Lemma \ref{lem: centralintegralrep}, it thus suffices to show that for any bounded convex $K$, there exists a Borel measurable set $K^{\prime}$ of equal Lebesgue measure such that 
\[ \prob \{ X + W \in K \} \leq \prob \{ \lambda X + W \in K' \}.\]
By conditioning on $X$, it suffices to show that 
\[  \prob \{ x + W \in K \} \leq \prob \{ \lambda x + W \in K' \},\]
for  arbitrary fixed $x$.

To that end, note that for a fixed bounded convex set $K$, an application of Pr\'ekopa-Leindler inequality shows that $p(x):=\prob \{ x + W \in K \}$ is a log-concave function of $x$. In particular, $p$ is unimodal. Denote by $x_0$ the point where the function $p$ achieves its maximum. By unimodality, we have, for any $\lambda \in [0,1],$
\[ p\left((1-\lambda) x_0+\lambda x \right) \geq p(x)\]
for an arbitrary fixed $x$. 

In other words, we have the following for any arbitrary fixed $x$:
\[  \prob \{ x + W \in K \} \leq \prob \{ (1-\lambda) x_0 +\lambda x + W \in K \}.\]
Setting $K'=K-(1-\lambda) x_0$, we thus obtain: 
\[ \prob \{ x + W \in K \} \leq \prob \{ \lambda x + W \in K' \},\]
as desired.
\end{proof}
Note that the geometric consequence of the previous result obtained by letting $X \sim \unif (K) $, $W \sim \unif (\ball)$ and $\alpha \to 0$, is trivial because $ \lambda K + \ball \subseteq K + \ball$ up to a translation. The next result is similar in this aspect. But, as before, the entropic version needs a little more work.

For an arbitrary contraction $T$, if it acts on a ball $\ball (0,r)$, then up to a shift, we would have $ T(\ball (0,r)) \subset \ball (0,r).$ Therefore, if both $X$ and $W$ are radially-symmetric and unimodal, heuristically, one may expect that, up to a shift, the distribution of $T(X)+W$ is more concentrated than the distribution of $X+W$, hence should have smaller R\'enyi entropies. We turn this intuition into the following theorem:

\begin{restatethis}{thm}{radsymunimodXW} \label{thm: radsymunimodXW}
For any radially-symmetric, unimodal, $\R^d$-valued random vectors $X$ and $W$, any contraction $T: \R^d \to \R^d$, we have
\[
f_{X + W} \preceq f_{T (X) + W },
\]
and consequently,
\[
h_{\alpha}(T(X) + W) \leq h_{\alpha} (X + W) \text{ for all } \alpha \in (0, \infty).
\]
\end{restatethis}
\begin{proof}
The density $f_{X + W}$ is already radially-symmetric and unimodal, so $f^{\ast}_{X+W} = f_{X+W}$. Given this, to show that $f_{X+W}\preceq f_{T(X)+W}$, by Lemma \ref{lem: centralintegralrep}, we need to produce for each $r >0 $, a measurable set $B'$ with $\vol (B') = \vol (\ball (0,r))$ satisfying $\prob \{ X + W \in \ball (0,r) \} \leq \prob \{ T(X) + W \in B' \}$ . Anderson's theorem \cite{Anderson55} implies that the function $x \mapsto \prob \{  x + W  \in \ball (0,r) \}$ is radially-symmetric and unimodal. Consequently, $\prob \{ x + W \in \ball (0,r) \}  \leq \prob \{ (T(x) - T(0)) +W \in \ball (0,r) \} $ for each fixed $x$, since $T$ is a contraction. By conditioning, 
\[
 \prob \{ X+W \in \ball (0,r) \} \leq \prob \{ \left(T(X)-T(0)\right)+ W \in \ball (0,r) \},
\]
and therefore, setting $B' = \ball (0,r) + T(0)$ does the job. 

\end{proof}
\begin{rem}
Let $K$ be a compact set with non-zero volume, $K^{\ast}$ the centered ball with same volume as $K$, and $T$ a contraction as before. Now the Brunn--Minkowski inequality implies $\vol (K + \ball ) \geq \vol (K^{\ast}  + \ball )$ in this case. Moreover, since $K^{\ast}$ is a ball, $T[K^{\ast}]-T(0) \subseteq K^{\ast}$ and hence $(T[K^{\ast}]-T(0))+ \ball \subseteq K^{\ast} + \ball$. Combining the two elementary observations we get, $\vol (T[K^{\ast}] + \ball) \leq \vol (K + \ball)$. This inequality for volumes can also be obtained as a corollary to the above theorem by applying it to the case when $X  \sim \unif (K), W \sim \unif (\ball)$, and using the observation due to Brascamp and Lieb that $h_{\alpha} (f \star g) \geq h_{\alpha} (f^{\ast} \star g^{\ast} )$ for $\alpha \in (0,1)$ \cite[Proposition 9]{BrascampLieb76} (in fact, the result of Brascamp and Lieb is true for all $\alpha \in [0, \infty]$ and any number of summands, see \cite{MadimanWang14}). 
\end{rem}

It is possible to trade the condition on the radial symmetry of $X$ for certain (stronger) linearity assumptions on $T$. Recall that random vector $X = (X_{1} , \ldots . X_{d})$ on $\R^d$ is said to be unconditional if all $(\pm X_{1} , \cdots , \pm X_{d})$ have the same distribution regardless of the choice of signs $\pm$. 

\begin{restatethis}{thm}{lcXunconditionalWdiagT} \label{thm: lcXunconditionalWdiagT}
Let $X$ be an $\R^d$-valued log-concave random vector, $W$ an $\R^d$-valued unconditional log-concave random vector. For any diagonally linear contraction $T$, we have 
\[
f_{X + W} \preceq f_{T(X) + W},
\]
and consequently,
\[
h_{\alpha}(T(X) + W) \leq h_{\alpha} (X + W) \text{ for all } \alpha \in (0, \infty).
\]
 
\end{restatethis}
\begin{proof}
We will denote the density of $X$ and $W$ by $f$ and $g$ respectively. Since $W$ is unconditional, we may assume that the diagonal elements $\lambda_{i}$ of $T$ are non-negative. Since both $X$ and $W$ are log-concave, $X+W$ is also log-concave and so the super-level sets $\{ f_{X + W} > t \}$ of its density are convex sets. By appealing to the equality case in Lemma \ref{lem: centralintegralrep}, it is sufficient to show, for every convex set $K$, there exists a Borel measurable set $K'$ having the same volume as $K$, such that 
\[
\int_{K} f_{X + W} \leq \int_{K'} f_{T(X) + W}. 
\] 
Upon explicitly writing out the convolution and  changing variables, the above reads
\[
\int_{K'} \int_{\R^d} f(y) g(x-T(y)) \d y \d x  \geq \int_{K} \int_{\R^d} f(y) g(x-y) \d y \d x,
\]
where $K'$ is a Borel measurable set having the same Lebesgue measure with $K$.\\
To this end, by Fubini's theorem, we have 
\begin{align*}
&\int_{K}  g(x-T(y)) \d x =\\
&\int_{\Pi_{e_1^{\perp}}(K)} \left( \int_{I(x_2,\cdots, x_d)} g(x_1-\lambda_1 y_1,x_2-\lambda_2 y_2,\cdots, x_d-\lambda_d y_d) \d x_1\right) \d x_2 \cdots \d x_d,
\end{align*}
where $\Pi_{e_1^{\perp}}(K)$ denotes the orthogonal projection of $K$ onto the orthogonal complement of $e_1$ and $I(x_2,\cdots, x_d)$ is the support of the inner integrand.
Note that, since $K$ is convex,  for fixed $x_2,\cdots, x_d$ and $y$, 
\[ I(x_2,\cdots, x_d) =[a(x_2,\cdots, x_d), b(x_2,\cdots, x_d)]\]
is an interval, and $g(z,x_2-\lambda_2 y_2,\cdots, x_d-\lambda_d y_d)$ is an even log-concave function in $z$. Hence,
\[ p(z):=\int_{I(x_2,\cdots, x_d)} g(x_1-z,x_2-\lambda_2 y_2, \cdots, x_d-\lambda_d y_d) \d x_1,\]
is a log-concave function on $\R$ whose maximum is attained at $$z_0= \frac{a(x_2,\cdots, x_d)+b(x_2,\cdots, x_d)}{2}.$$  
We deduce that 
\[  p((1-\lambda_1)z_0 + \lambda_1 y_1) \geq  p(y_1).\]
Whence we have 
\[ \int_{I(x_2,\cdots, x_d)-(1-\lambda_1)z_0} g(x_1-\lambda_1 y_1,x_2-\lambda_2 y_2 , \cdots, x_d-\lambda_d y_d) \d x_1 \]
\[\geq \int_{I(x_2,\cdots, x_d)} g(x_1-y_1,x_2-\lambda_2 y_2, \cdots, x_d-\lambda_d y_d) \d x_1.\]
To summarize, we have shown that 
\begin{align*}
\int_{S_{e_1}^{\lambda_1}(K)} g(x_1-\lambda_1 y_1 ,x_2-\lambda_2 y_2, &\cdots, x_d-\lambda_d y_d) \d x \geq\\
&\int_{K} g(x_1-y_1,x_2-\lambda_2 y_2 ,\cdots, x_d-\lambda_d y_d) \d x,
\end{align*}
where $S_{e_i}^{\lambda_i}(K)$ are defined as the following:
\[\left\{ x \times \left\{\left[-\frac{t_2-t_1}{2}, \frac{t_2-t_1}{2}\right]+\lambda_i\frac{(t_2+t_1)}{2}\right\}: x \in e_i^{\perp}, (x,te_i)\cap K = x \times [t_1,t_2] \right\}.\]
It is worth noting that, $S_{e_i}^{\lambda_i}(K)$ is a member of the so-called shadow system of $K$ along the direction $e_i$, which was introduced in \cite{RS}. It is well-known that shadow systems preserve convexity, however we provide a proof here for completeness. Let $u$ be any unit vector, note that the convex body $K$ can be expressed as 
\[ K=\{(x,s u): x\in u^{\perp},  \ g(x) \leq s \leq f(x)\},\]
where $g(x)$ is a convex function and $f(x)$ is a concave function.
Therefore, $S_{u}^{\lambda}(K)$ is
\[ \left\{ x \times \left\{\left[-\frac{f(x)-g(x)}{2}, \frac{f(x)-g(x)}{2}\right]+\lambda \frac{g(x)+f(x)}{2}\right\}: x \in u^{\perp}, \lambda \in [0,1] \right\}.\]
Since $\lambda \in [0,1],$ we have that $-\frac{f(x)-g(x)}{2}+\lambda \frac{g(x)+f(x)}{2}$ is convex, while  $\frac{f(x)-g(x)}{2}+\lambda \frac{g(x)+f(x)}{2}$ is concave. This establishes the convexity of $ S_{u}^{\lambda}(K)$.
Moreover, by Fubini's theorem, $S_{e_i}^{\lambda_i}$ also preserves the volume:
\[ \vol(S_{e_i}^{\lambda_i}(K))=\vol(K), \quad i=1,\cdots, d.\]
Repeating the argument coordinate-wise, we have 
\begin{align*}
\int_{S_{e_1}^{\lambda_1} S_{e_2}^{\lambda_2} \cdots S_{e_d}^{\lambda_d} (K)} &g(x_1-\lambda_1 y_1,x_2-\lambda_2 y_2, \cdots, x_d-\lambda_d y_d) \d x \geq\\
&\int_{K} g(x_1-y_1,x_2-y_2,\cdots, x_d-y_d) \d x.
\end{align*}
Choose $K'=S_{e_1}^{\lambda_1} S_{e_2}^{\lambda_2} \cdots S_{e_d}^{\lambda_d} (K)$, the desired result now follows.
\end{proof}

If in addition $W$ is radially symmetric, then  rotational invariance enables us to generalize Theorem \ref{thm: lcXunconditionalWdiagT} to any affine contraction $T$. 
\begin{restatethis}{cor}{lcXradsymWaffineT} \label{cor: lcXradsymWaffineT}
If $X$ is a log-concave random vector and $W$ is a radially-symmetric log-concave random vector, then for any affine contraction $T$, we have
\[
f_{X +W} \preceq f_{T(X) + W },
\]
and consequently,
\[ 
h_{\alpha}(T(X)+W) \leq h_{\alpha}(X+W) \text{ for all } \alpha \in (0, \infty).
\]
\end{restatethis}
\begin{proof}
First, note that the majorization order remains invariant under orthogonal transformations (see Definition \ref{def: majorisation} and Equation \eqref{f*}).
By polar factorization and further diagonalization of the symmetric component, we can write $T = Q_{1} \Lambda Q_{2}$ for orthogonal matrices $Q_{1}, Q_{2}$ and diagonal $\Lambda$. Using Theorem \ref{thm: lcXunconditionalWdiagT},
\[
\begin{split}
    f_{T(X) + W } &= f_{Q_{1} \Lambda Q_{2} (X) + W} = f_{Q_{1} \Lambda Q_{2} (X) + Q_{1} W} \simeq f_{\Lambda Q_{2} (X) +  W} \succeq f_{ Q_{2} (X) +  W} \\
    & = f_{ Q_{2} (X) +  Q_{2} W} \simeq f_{X + W}.
\end{split}
\]

\end{proof}
By letting $\alpha \to 0$, we obtain the following inequality for convex bodies. 

\begin{cor} \label{cor: convexKlinearT}
Let $K$ be a convex body and $r>0$. Then,
\[
\vol(T(K) + r\ball) \leq \vol(K+ r\ball),\]
for any affine contraction $T$.

\end{cor}

The above corollary can also be seen as a consequence of the fact that intrinsic volumes decrease under linear contractions \cite[Proposition 1.1]{PaourisPivovarov13}. In fact, \cite[Proposition 1.1]{PaourisPivovarov13} can also be deduced from our method as shown below. 

Let the $i$-th intrinsic volume of a convex body $K$ be denoted by $V_{i}(K)$. For the definition and properties of the $V_{i}(K)$, we refer the interested reader to the standard textbook \cite{Schneider14}.
\begin{cor} \label{cor: intrinsicvolumeslinearcontractions}
Let $K$ be a convex body in $\R^d$ and $T: \R^d \to \R^d$ a linear contraction. Then, $V_{i} (T[K]) \leq V_{i} (K)$ for $i = 0 , \cdots , d$.
\end{cor}
\begin{proof}
 Recall that intrinsic volumes are invariant under orthogonal transformations. By polar factorization and subsequent diagonalization (see the proof of Corollary \ref{cor: lcXradsymWaffineT}), we can assume that $T$ is a diagonally linear contraction of the form $T = T_{d} \cdots T_{1}$, where $T_{j}$ has $(1, \ldots, \lambda_{j}, \ldots, 1)$ on its diagonal with $ \lambda_{j} \in [0,1].$
 By convexity of intrinsic volumes for shadow systems \cite[Section 2]{Shephard64}, $V_{i} (S^{\lambda}_{e_{1}} (K) )$ is a convex  function of $\lambda \in [-1,1]$. Since $S^{1}_{e_{1}} (K) = K$ and $S^{-1}_{e_{1}} (K)$ is the reflection of $K$ across $e^{\perp}_{1}$, it follows $V_{i} (S^{1}_{e_{1}} (K)) = V_{i} (S^{-1}_{e_{1}} (K)).$ Therefore, we have $V_{i} (S^{\lambda_{1}}_{e_{1}} (K)) \leq V_{i} (K)$. From the definition of $S^{\lambda_{j}}_{e_{j}}(K)$ appearing in the proof of Theorem \ref{thm: lcXunconditionalWdiagT}, it is straightforward to check that $T_{1}[K] \subseteq S^{\lambda_{1}}_{e_{1}} (K).$ The monotonicity of intrinsic volumes then allows us to conclude $V_{i}(T_{1}[K]) \leq V_{i}(K)$. By the same argument for each $\lambda_{j}$ and $e_{j}$,
 \[
V_{i}(K) \geq V_{i}( T_{1} [K] ) \geq V_{i}(T_{2}T_{1}[K]) \geq \cdots \geq V_{i} (T_{d} \cdots T_{1}[K]) = V_{i} (T[K]).
 \]

\end{proof}
Further, if we impose the stronger restriction of unconditionality on $X$, then Corollary \ref{cor: lcXradsymWaffineT} also holds when $T$ is a strong contraction. 

\begin{restatethis}{thm}{unconditionalXWstrongT} \label{thm: unconditionalXWstrongT}
Let $X, W$ be two unconditional log-concave random vectors. Then for any strong contraction $T$, we have
\[
f_{X + W} \preceq f_{T(X) + W},
\]
and consequently,
\[
h_{\alpha}(T(X) + W) \leq h_{\alpha} (X + W) \text{ for all } \alpha \in (0, \infty).
\]
\end{restatethis}

\begin{proof}
Denote by $f$ and $g$ the densities of $X$ and $W$, respectively. Note that R\'enyi entropy is translation invariant, by subtracting $T(0)$ from $T$, we may assume that $T(0)=0$. Since both $X$ and $W$ are unconditional log-concave, $r(x)$, the density of $X+W$, is also unconditional log-concave, whose super level sets, therefore, are unconditional convex sets. Again by Lemma \ref{lem: majorisationimpliesconvex} and Lemma \ref{lem: centralintegralrep}, it suffices to show that for any unconditional convex set $K$, one has
\[ \int_{K} \int_{\R^d} f(y) g(x-T(y)) \d y \d x  \geq \int_{K} \int_{\R^d} f(y) g(x-y) \d y \d x.\]
It suffices to show that 
\[ \int_{K}  g(x-T(y)) \d x  \geq \int_{K} g(x-y) \d x.\]
To this end, by Fubini's theorem, we have 
\begin{align*}
&\int_{K}  g(x-T(y)) \d x =\\
&\int_{\Pi_{e_1^{\perp}}(K)} \left( \int_{I(x_2,\cdots, x_d)} g(x_1-T_1(y),x_2-T_2(y),\cdots, x_d-T_d(y)) \d x_1\right) \d x_2 \cdots \d x_d,
\end{align*}
where $T$ is represented as $(T_{1}, \ldots, T_{d})$.
Note that, for fixed $x_2,\cdots, x_d$ and $y$, $I(x_2,\cdots, x_d)$ is a symmetric interval and $g(z,x_2-T_2(y),\cdots, x_d-T_d(y))$ is an even log-concave function in $z$. Hence,
\[ p(z):=\int_{I(x_2,\cdots, x_d)} g(x_1-z,x_2-T_2(y),\cdots, x_d-T_d(y)) \d x_1,\]
is an even log-concave function on $\R$. Now by the strong contractivity of $T$, we deduce that 
\[  p(T_1(y))=p(T_1(y)-T_1(0)) \geq  p(y_1).\]
Whence we have 
\[ \int_{I(x_2,\cdots, x_d)} g(x_1-T_1(y),x_2-T_2(y)\cdots, x_d-T_d(y)) \d x_1 \]
\[\geq \int_{I(x_2,\cdots, x_d)} g(x_1-y_1,x_2-T_2(y)\cdots, x_d-T_d(y)) \d x_1.\]
To summarize, we have shown that 
\begin{align*} 
\int_{K} g(x_1-T_1(y),x_2-T_2(y),&\cdots, x_d-T_d(y)) \d x \geq\\
&\int_{K} g(x_1-y_1,x_2-T_2(y),\cdots, x_d-T_d(y)) \d x.
\end{align*}
Repeat the argument coordinate-wise, we have 
\begin{align*}
\int_{K} g(x_1-T_1(y),x_2-T_2(y),&\cdots, x_d-T_d(y)) \d x \geq \\
&\int_{K} g(x_1-y_1,x_2-y_2,\cdots, x_d-y_d) \d x,
\end{align*}
as desired.
\end{proof}
\begin{cor} \label{cor: unconditionalconvexKL}
Let $K,L$ be two unconditional convex bodies, then
\[
\vol(T(K) + L) \leq \vol(K+L)\]
holds for any strong contraction $T$.
\end{cor}
Maps of the form $\grad \varphi$, for convex $\varphi$, play an important role in geometry via the theory of optimal transport where such maps solve the mass transport problem for the quadratic cost (see for example, \cite[Chapter 2]{Villani03} ). Moreover, they play the role of the positive semi-definite matrices in a far reaching generalization of polar factorization of matrices to maps $\R^d \to \R^d$ due to Brenier \cite{Brenier91}. We think it is worthwhile to note that a result for contractions of this form can be obtained if we assume $W$ to be radially-symmetric and log-concave. 
\begin{cor} \label{cor: unconditionallcXradsymlcWBrennierT}
If $X$ is unconditionally log-concave, $W$ is radially-symmetric log concave, and $T=\nabla \varphi$ for some smooth convex function $\varphi$ on $\R^d$, is a contraction. Then, we have
\[
f_{X + W} \preceq f_{T(X) + W},
\]
and consequently,
\[ 
h_{\alpha}(T(X)+W) \leq h_{\alpha}(X+W) \text{ for all } \alpha \in (0, \infty).
\]
\end{cor}

\begin{proof}
 Replacing $\varphi(x)$ with $\varphi(x)-\langle x, \nabla \varphi(0)\rangle$, if necessary, we may assume that $\nabla \varphi(0)=0$. Since $\varphi$ is a smooth convex function, apply the multi-dimensional mean-value theorem, we have 
 \[ \nabla \varphi(y)-\nabla \varphi(0)=\left(\int_0^1 D T(ty) \d t \right) \cdot y.\]
 For each $y \in \R^d,$ 
\[H(y):=\int_0^1 D T(ty) \d t,\]
as a convex combination of positive semi-definite matrices, is again positive semi-definite. Therefore, there exists an orthogonal matrix $Q_y,$   such that $ H(y)= Q_y^{\perp} \Lambda_y Q_y$, where $\Lambda_y$ is a diagonal matrix for every $y$. Then for any symmetric convex set $K$,
\[  
\int_{K}  g(x-\nabla \varphi(y)) \d x  = \int_{K}  g(x-H(y) \cdot y ) \d x = \int_{K}  g(x-(Q_y^{\perp} \Lambda_y Q_y)\cdot y )\d x.\]
By radial-symmetry of $g$ and Theorem \ref{thm: unconditionalXWstrongT} applied to $\Lambda_{y}$, we have 
\[ \int_{K}  g(x-(Q_y^{\perp} \Lambda_y Q_y) y )\d x =\int_{Q_y(K)}  g(x- \Lambda_y \cdot ((Q_y) y) )\d x \]
\[\geq \int_{Q_y(K)}  g(x- (Q_y) \cdot y )\d x =\int_{K}  g((Q_y) \cdot x- (Q_y) \cdot y )\d x = \int_{K}  g(x- y) \d x,\]
as desired.

\end{proof}

\section{Gaussian noise} \label{sec: gaussian noise}

The entropy power inequality (EPI), 
\begin{equation} \label{EPI}
N(X+ Y) \geq N(X) + N(Y),
\end{equation}
for $\R^d$-valued random vectors $X,Y$ with density, is a fundamental result of information theory occupying a place akin to the Brunn-Minkowski inequality in convex geometry. Here $N(X) :=e^{\frac{2 h(X)}{d}}$ denotes the entropy power of $X$. It was shown by Costa \cite{Costa85} that this inequality improves when one of the random variables involved is Gaussian. More precisely, he observed that $N(X + \sqrt{t}Z)$ is concave as a function of $t$. Later, a vector-generalization\footnote{here ``vector-generalization'' refers to the usage of a matrix instead of the ``$t$'' in Costa's EPI} of Costa's inequality was published by Liu, Liu, Poor and Shamai \cite{LiuLiuPoorShamai10}. However, a flaw in their proof and a partial resolution was discovered by Courtade, Han and Wu \cite{CourtadeHanWu17}. Thankfully, we will only need the ``correct part'' of the vector-generalization. 
\begin{thm} \cite{LiuLiuPoorShamai10, CourtadeHanWu17} \label{thm: vectorepi}
Let $Z \sim \normal(0 , \Sigma)$ be a Gaussian random vector in $\R^d$, and let $X$ be any random vector in $\R^{d}$ having density with respect to the Lebesgue measure. For the $d \times d$ identity matrix $I_{d}$, suppose $S$ is a positive semi-definite matrix such that $I_{d} - S$ is also positive semi-definite and $S$ commutes with $\Sigma$. Then, 
\[
N(X + S^{1/2}Z ) \geq \det (I_{d} - S)^{1/d} N(X) + \det S^{1/d} N(X + Z).
\] 
\end{thm} 
Using these results, we can solve the Gaussian version of the Kneser-Poulsen problem for the Boltzmann-Shannon entropy in a strong sense when the contraction $T$ is linear. 
\begin{restatethis}{thm}{arbitXgaussianZlinearT} \label{thm: arbitXgaussianZlinearT}
Let $Z \sim \normal (0, I_{d})$ be the standard Gaussian random vector in $\R^d$, and let $X$ be a random vector in $\R^d$ having density with respect to the Lebesgue measure, and $T$ a linear contraction. Then we have 
\[ N(X+Z) \geq N(T(X)+Z)+(1-\lip^2(T))N(X).\]
\end{restatethis}
\begin{proof}
By polar decomposition, there exists an orthogonal matrix $Q$ and a positive semi-definite matrix $A$ such that $T=QA$. Following the invariance of the entropy power and the invariance of the standard Gaussian distribution under the action of the orthogonal group, we have $N(T(X) + Z) = N(A(X) + Z)$. Moreover, since $T$ is contractive, $I_{d}-A^2$ is positive semi-definite. Therefore, by the vector version of the EPI (\ref{thm: vectorepi}) applied to $S = A^{2}$ and $A(X)$ instead of $X$, we have
\[ N(A(X)+A(Z)) \geq \det(I_{d}-A^2)^{1/d} N(A(X))+\det(A^2)^{1/d}N(A(X)+Z).\]

Assume first that $T$ is non-singular so that $A$ is positive-definite. Then an application of the change of variables formula $h(A(X)) = h(X) + \E \log \det A'(X)$ yields
\[ N(X+Z) \geq \det (I_{d}-A^2)^{1/d} N(X)+ N(A(X)+Z).\]

Denote by $\lambda_1(A)$ the largest eigenvalue of $A$. We have that all eigenvalues of $I_d-A^2$ are greater than $1-\lambda_1^2(A).$ The desired inequality now follows from the fact that $\lambda_1(A)$ is also equal to the Lipschitz constant of the linear map $T$. 

If $T$ is singular, then by the argument above, we may assume that $T$ is positive semi-definite. Therefore, there exists a sequence of positive definite matrices $T_{\epsilon}$ such that 
\[\lim_{\epsilon \rightarrow 0} \Vert T_{\epsilon}-T \Vert_{op}=0,\]
where $\Vert \cdot \Vert_{op}$ denotes the operator norm. This sequence can be obtained, for example, by considering $T_{\epsilon} = (1- \epsilon)T + \epsilon I_{d}$. For such a $T_{\epsilon}$, our earlier result implies that
\[N(X+Z) \geq N(T_{\epsilon}(X)+Z)+(1-\Vert T_{\epsilon} \Vert_{op}^2)N(X).\]

Since $\Vert T_{\epsilon} \Vert_{op} \rightarrow \Vert T \Vert_{op}$, as $ \epsilon \rightarrow 0,$  all we need to show is that 
\[ \liminf_{\epsilon \rightarrow 0} N(T_{\epsilon}(X)+Z) \geq N(T(X)+Z)  .\]
Denote by $g(x)$, $f(x)$, and $f_{\epsilon}(x)$ the densities of $X$, $T(X)+Z$, and $T_{\epsilon}(X)+Z$, respectively. One has that
\begin{align*}
    f(x) &= \left(\frac{1}{2 \pi}\right)^{d/2}\int_{\R^d} g(y) e^{-\frac{\Vert x- T(y) \Vert^2}{2}} \d y,
\end{align*}
\text{ and }  
\begin{align*}f_{\epsilon}(x) = \left(\frac{1}{2 \pi}\right)^{d/2}\int_{\R^d} g(y) e^{-\frac{\Vert x- T_{\epsilon}(y) \Vert^2}{2}} \d y.   
\end{align*}

By the Dominated Convergence Theorem, we have
\[  f_{\epsilon}(x) \rightarrow f(x) ,  \quad \text{as} \ \epsilon \rightarrow 0.\]
Note that for any positive integer $d$,
 \[ 0< f_{\epsilon}(x) \leq \left(\frac{1}{2 \pi}\right)^{d/2} <1.\]
Therefore,
  \[ f_{\epsilon}(x) \log \frac{1}{f_{\epsilon}(x)} >0.\]
  
By Fatou's lemma, one has
\[\liminf_{\epsilon \rightarrow 0} h(T_{\epsilon}(X)+Z) \geq h(T(X)+Z),\]
whence the desired inequality now follows.
\end{proof}
The concavity of entropy power, used in the above theorem, is a deep result in information theory and related fields. By using an arguably simpler fact, namely the entropy maximization property of Gaussians, and a pinch of linear algebra, a result of the same form under different hypotheses is obtained below.
\begin{restatethis}{thm}{gaussianGZstrongT} \label{thm: gaussianGZstrongT}
Let $Z \sim \normal (0, I_d)$ be the standard Gaussian random vector in $\R^d$, and let $G\sim \normal (\mu, \Lambda)$ be any Gaussian random vector in $\R^d$ with diagonal covariance matrix and $T$ any strong contraction. Then we have 
\[ N(G+Z) \geq N(T(G)+Z)+(1-\lip^2 (T))N(G).\]
If, in addition, $\Lambda=\alpha I_d$ for some $\alpha >0,$ the inequality holds for any contraction $T.$
\end{restatethis}

\begin{proof}

A direct calculation reveals that 
\[ N(G+Z)= 2 \pi e \det(I_d+\Lambda)^{1/d}, \quad N(G)= 2 \pi e \det(\Lambda)^{1/d}.\]
Since $(1-\lip^2 (T))\Lambda$ is positive semi-definite, by the fact that $\det(A)^{1/d}$ is concave on the set of positive semi-definite matrices, one has
\[ \det(I_d+\Lambda)^{1/d} \geq \det(I_d+\lip^2(T)\Lambda)^{1/d}+ \det((1-\lip^2 (T))\Lambda )^{1/d}.\]

Denote by $\Sigma$ the covariance matrix of $T(G)$, by the well-known fact that Gaussians maximize entropy under second-moment constraints, one has
\[  N(T(G)+Z) \leq 2 \pi e \det(I_d+ \Sigma)^{1/d}.\]
Therefore, to prove the desired inequality, it suffices to show that 
\[ \det(I_d+ \Sigma) \leq \det(I_d+\lip^2(T)\Lambda).\]

Let $\lambda_1, \cdots, \lambda_d$ be the eigenvalues of $\Sigma$, and $a_{11}, \cdots, a_{nn}$ be the diagonal entries of $\Sigma$. The first observation is that 
\[ a_{ii} \leq \lip^2(T) \Lambda_{ii}.\]

To see this, note that 
\[ a_{ii}=\E[T_i(G)-\E(T_i(G))]^2 \leq \E[T_i(G)-T_i(\E(G))]^2 \leq \lip^2(T) \Lambda_{ii},\]
where the first inequality  follows from the property of variance, while the second follows from the strong contractivity of $T$.

On the other hand
\[  \det(I_d+ \Sigma)= \prod_{i=1}^{d} (1+\lambda_i) \leq \prod_{i=1}^d (1+a_{ii})\leq  \prod_{i=1}^{d} (1+\lip^2(T)\Lambda_{ii})=  \det(I_d+\lip^2(T)\Lambda),\]
where the first inequality is Hadamard's inequality for positive semi-definite matrices.
If, in addition, $\Lambda=\alpha I_d$, by the AM-GM inequality, one has 
\[ \det(I_d+ \Sigma)^{1/d}=\prod_{i=1}^{d} (1+\lambda_i)^{1/d} \leq 1+ \frac{\text{Tr}(\Sigma)}{d}.\]

By the property of variance, we have
\begin{align*}
\text{Tr}(\Sigma) &= \E\vert|T(G)-\E(T(G))\vert|_2^2\\
  &\leq \E \vert|T(G)-T(\E(G))\vert|_2^2 \\ &\leq \lip^2(T) \text{Tr}(\Lambda) \\
  &= \lip^2(T)  \alpha d.
\end{align*}
Whence,
\[\det(I_d+ \Sigma)^{1/d} \leq 1+ \lip^2(T)\alpha = \det(I_d+\lip^2(T)\Lambda)^{1/d},\]
as desired.
\end{proof}

Observe that the proof goes through for any contraction $T$ if $G$ is an isotropic Gaussian. This particular case can be generalized to allow for $G$ here to be replaced by an isotropic log-concave distribution $X$ with conditions on $T$ that are directly related to a measure of the distance of $X$ from being Gaussian. 

\begin{restatethis}{thm}{isotropiclcXgaussianZ} \label{thm: isotropiclcXgaussianZ}
Let $X$ be an isotropic log-concave random vector in $\R^d$, and let $Z$ be standard Gaussian. Let $\Delta (X) = \frac{h(Z_{X}) - h(X)}{d}$, where $Z_{X}$ is a centered Gaussian in $\R^d$ having the same covariance matrix as $X$. Then, for any contraction $T: \R^d \to \R^d$, we have 
\[
N(X + Z) \geq N(T(X) + Z)+(1-(e^{ \Delta(X)} \lip(T))^2)N(X).
\]
\end{restatethis}
\begin{proof}
Suppose the covariance matrix of $X$ is $\Sigma_{X} =\alpha I_d$, for some $\alpha >0.$  By the definition of $\Delta(X)$, we have the following equality on the entropy of $X$:
\[
h(X) = \frac{1}{2} \log \det \Sigma_{X} + \frac{d}{2} \log (2 \pi e) - \Delta(X)d.
\]

In terms of the entropy power $N(X) = e^{2h(X)/d}$, this reads as follows:
\[
N(X) =  e^{-2\Delta(X)} 2 \pi e (\det \Sigma_{X})^{1/d}=e^{-2\Delta(X)} 2 \pi e \alpha.
\]
By entropy power inequality, we have: 
\[ N(X+Z) \geq  N(X)+N(Z) \geq 2 \pi e(1+e^{-2\Delta(X)}\alpha).\]

On the other hand, for a Lipschitz map $T$, again by the fact that Gaussian random variables maximize entropy under second-moment constraints, and the AM-GM inequality, one has 
\[
N(T(X) + Z) \leq 2 \pi e  (\det (\Sigma_{T(X)}+I_d))^{1/d} \leq 2 \pi e (1+\alpha \lip^2(T)).
\] 
The desired inequality now follows.
\end{proof}

Bobkov and Madiman \cite[Theorem 2]{BobkovMadiman10} showed, building upon the work of Keith Ball \cite{Ball88}, that the hyperplane conjecture in convex geometry is equivalent to the existence of a constant $C$ such that $\Delta (X) \leq C$ for any log-concave random vector $X$ in any dimension. A crucial observation in \cite{BobkovMadiman10} is that the isotropic constant of an $\R^d$-valued log-concave random vector $X$, $L_{X} = \sqrt h_{\infty}^{1/d} \det \sigma_{X}^{1/2d}$, is related to $\Delta (X)$ via
\[
\log \left( \sqrt{ \frac{2 \pi}{e}} L_{X} \right) \leq \Delta (X).
\]
Therefore, one can also easily write a condition for $T$ in terms of the isotropic constant under which the desired entropic inequality holds. 

\begin{cor} \label{cor: isotropiclcXgaussianWsomeT}
Suppose $X$ is an isotropic log-concave random vector in $\R^d$ with isotropic constant $L_{X}$. Then for any Lipschitz map $T: \R^d \to \R^d$ with $\lip (T) \leq \sqrt{\frac{e}{2 \pi L_{X}^2}}$, we have 
\[
h(T(X) + Z ) \leq h(X + Z). 
\]
\end{cor} 

\section{R\'enyi entropy of order two} \label{sec: h2}
A more direct analysis of metric properties of R\'enyi entropy is sometimes possible if one uses perturbations of R\'enyi entropies where the metric makes an explicit appearance. The family of perturbations that we use in this section are called diversities. They were introduced in the context of quantification of biodiversity by Leinster and Cobbold \cite{CobboldLeinster12}. Mathematical aspects of this notion were developed further by Leinster and Meckes \cite{LeinsterMeckes16}, Leinster and Roff \cite{LeinsterRoff21}, and Madiman, Meckes, with a subset of the present authors \cite{AishwaryaLiMadiman22}.

\begin{definition} \label{def: diversity}
Let $(X, d)$ be a metric space and $\mu \in \mathcal{P}(X)$. The diversity of order $\alpha$ of $\mu$ is defined by 
\[
D^{K}_{\alpha}(\mu) =
\begin{cases}
\left( \int_{X} (1/K \mu)^{1-\alpha} \d \mu \right) ^{1/(1-\alpha)}, & \textnormal{ if } \alpha \in [0,1), \cup (1, \infty) \\
e^{-\int_{X} \log K \mu \d \mu }, & \textnormal{ if } \alpha = 1, \\
\frac{1}{\essup_{\mu} K \mu}, & \textnormal{ if } \alpha = \infty, \\
\end{cases}
\]
where $K\mu (x) = \int_{X} e^{-d(x,y)} \d \mu(y)$.
\end{definition}
\begin{rems}
\begin{enumerate}
\item[]
\item It is $\log D^{K}_{\alpha}$ which corresponds to R\'enyi entropy-like quantities. 
\item Instead of $e^{-d(x,y)}$ one could use other ``kernels'' $K(x,y)$. This explains the superscript $K$ of $D^{K}_{\alpha}$.
\item If $(X,d)$ is a metric space then so is $(X,td)$ for $t>0$. We will denote the corresponding diversity measures by $D^{K^{t}}_{\alpha}$ or simply $D^{t}_{\alpha}$ when there is no scope for confusion.
\end{enumerate}
\end{rems}
 
For a metric structure on a finite set $X$, all integrals in the definition are just finite sums. From here one easily takes limits term-by-term to see that $\lim_{t \to \infty}D^{t}_{\alpha}(\mu) = e^{H_{\alpha}(\mu)}$, where $H_{\alpha}(\mu) = \frac{1}{1- \alpha} \log \sum_{x \in X} \mu (x)^{\alpha}$. Thus, we see that all R\'enyi entropies can be recovered in this case. Here we are interested in the space $X = \R^d$ equipped with the standard Euclidean metric. In this setup too, one can recover R\'enyi entropies (see \cite[Proposition 2.9]{AishwaryaLiMadiman22}). However, all we need is the special case below.

\begin{lem}
Let $\mu$ be a probability measure on $\R^d$ with density with respect to the Lebesgue measure. Then, 
\[
\lim_{t \to \infty} C_{d} \frac{D^{t}_{2}(\mu)}{t^{d}} = e^{h_{2}(\mu)},
\]
for a constant $C_{d}$ only depending on the dimension $d$.
\end{lem}

Therefore we get inequalities for the R\'enyi entropy $h_{2}$ if we prove the corresponding inequalities for $D^{t}_{2}$. 

\begin{restatethis}{thm}{h2} \label{thm: h2}
Let $X$ be an $\R^{d}$-valued random vector, and $W$ an $\R^d$-valued log-concave random vector with radially-symmetric density. Then, for any contraction $T: \R^d \to \R^d$, $t>0$, we have
\[
D^{t}_{2}(T(X) + W) \leq D^{t}_{2}(X + W).
\]
Consequently,
\[
h_{2}(T(X) + W) \leq h_{2}(X + W).
\]
\end{restatethis}
\begin{proof}
Note that $D^{t}_{2}(Y)$, for any random vector $Y$, can be probabilistically written as
\[
D^{t}_{2}(Y) = \left( \E_{Y,Y'} e^{- t \Vert Y - Y' \Vert} \right)^{-1},
\]
where $Y'$ is an independent copy of $Y$. To create an independent copy of $X + W$ we take $X' + W'$, where $X',W'$ are independent copies of $X,W$ respectively and $X,W,X',W'$ are all independent. Now,
\begin{align*}
D^{t}_{2}(X + W) &= \left( \E_{X,W, X', W'} e^{-t\|(X+W)-(X'+W')\|} \right)^{-1} \\
& = \left( \E_{X,X'} \E_{W,W'} e^{-t\|(X-X')-(W'-W)\|} \right)^{-1}.
\end{align*}
Similarly, 
\[
D^{t}_{2}(T(X) + W) = \left( \E_{X,X'} \E_{W,W'} e^{-t\|(T(X)-T(X'))-(W'-W)\|} \right)^{-1}.
\]

Let $g$ denote the density of $W'-W$. Being the density of a convolution of two radially-symmetric log-concave densities (recall here that $-W$ has the same distribution as $W$), it is radially-symmetric and log-concave. Moreover, the function $\phi(z) = e^{-\Vert z \Vert}, z \in \R^{d},$ is also radially-symmetric, log-concave, and consequently so is $\phi \star g$. Observe that, for any fixed values $X = x, X'=x'$,
\[
\begin{split}
& \E_{W,W'} e^{-t\|(T(x)-T(x'))-(W'-W)\|} = \phi \star g (T(x)-T(x')) \\
& \E_{W,W'} e^{-t\|(x-x')-(W'-W)\|} = \phi \star g (x-x').
\end{split}
\]
Then, since the value of a radially-symmetric log-concave function is higher for points closer to the origin, from $\Vert T(x) - T(x') \Vert \leq \Vert x  - x' \Vert$ we get
\[
\E_{W,W'} e^{-t\|(T(x)-T(x'))-(W'-W)\|} \geq \E_{W,W'} e^{-t\|(x-x')-(W'-W)\|}. 
\]
Now the desired result follows by taking expectations. 
\end{proof}

Here we remind the reader that when $W \sim \unif ( \ball ) $, the theorem above is trivial since it follows from the intersection conjecture for two balls. Yet, \textit{a priori}, it gives monotonicity under contraction for the functional $$L \mapsto \sup_{\substack{X \in L \\ W ~\textnormal{l.c.}, \in \ball}} h_{2} (X + W),$$ which is closer to the volume $L \mapsto \vol(L + \ball)$ than the functional $$L \mapsto \sup_{X \in L} h_{2} (X + \unif (\ball))$$ obtained by the trivial entropic inequality. Here $X \in L$ means that the random vector $X$ takes values in the set $L$, while $W \, \textnormal{l.c.}, \in \ball$ means that $W$ is log-concave and takes values in $\ball$. 

If $X$ is assumed to have a log-concave density then the main result of this section can be extended to a comparison for all R\'enyi entropies at the cost of a constant which unfortunately blows up at the order $0$ thus not giving a direct inequality for volume. The main ingredient is \cite[Lemma 2.4]{MadimanMelbourneXu17} (for a proof, see \cite[Corollary 7.1]{FradeliziLiMadiman20}, or \cite[Corollary 4]{BialobrzeskiNayar21} for an alternate proof), which says that for an $\R^d$-valued log-concave random vector all R\'enyi entropies can be compared via
\[
h_{\beta}(X) - h_{\alpha}(X) \leq d \left( \frac{\log \beta}{\beta - 1} - \frac{\log \alpha}{\alpha - 1}  \right),
\]
for $\alpha \geq \beta > 0$.
Combined with the fact that R\'enyi entropies of a fixed random vector are non-increasing as a function of order, we obtain the following corollary. 
\begin{cor} \label{cor: lccomparisonfromh2}
Let $X$ be an $\R^{d}$-valued log-concave random vector, $W$ an $\R^{d}$-valued radially-symmetric log-concave random vector, and $T: \R^d \to \R^{d}$ a contraction. Then, 
\[
h_{\alpha} (T(X) + W ) \leq h_{\alpha} (X + W) + \sgn(2- \alpha)  \left(\frac{\log \alpha}{\alpha - 1} - \log 2 \right) d,
\]
for $\alpha > 0$. 
In particular, 
\[
h(T(X) + W) \leq h(X + W) + 0.307 d,
\]
or in other words,
\[
N(T(X) + W) \leq 1.85 N(X + W).
\]
\end{cor}
\qed

\subsection*{Acknowledgements} 
We thank Mokshay Madiman for helpful discussions and encouragement, especially during early stages of this project. The authors would like to express their gratitude to Shiri Artstein-Avidan for many valuable comments on the initial draft of this paper. Gautam Aishwarya was supported by ISF Grant 1750/20. Dongbin Li was supported by the UDRF Strategic Initiatives Grant. We are grateful to the organizers of the \textit{Workshop in Convexity and geometric aspects of Harmonic Analysis 2019} where the collaboration among the present authors was initiated. We are also grateful to an anonymous referee whose suggestions improved the presentation of the paper.

\bibliographystyle{amsplain}
\bibliography{shoulderofgiants}

\end{document}